\documentclass[12pt,leqno]{amsart}  

\usepackage{calc}
\usepackage[left=1in,right=1in,top=1in,bottom=1in]{geometry}  
\usepackage{graphicx}
\usepackage{amsmath,amssymb,epsfig,mathtools}
\usepackage{amsthm}
\usepackage{enumerate}
\usepackage{caption}
\usepackage{tikz}
\makeatletter
\renewcommand*\env@matrix[1][*\c@MaxMatrixCols c]{%
  \hskip -\arraycolsep
  \let\@ifnextchar\new@ifnextchar
  \array{#1}}
\makeatother

\newcommand{\Z}{{\mathbb Z}}
\newcommand{\ZZ}{{\mathbb Z}}

\newcommand{\Gal}{{\mathrm{Gal}}}

\newcommand{\Sel}{{\mathrm{Sel}}}

\newcommand{\val}{{\mathrm{val}}}

\newcommand{\p}{{ \mathfrak{p} }}

\newcommand{\Frob}{{ \mathrm{Frob} }}

\newcommand{\FF}{\mathbb{F}}

\newcommand{\QQ}{\mathbb{Q}}
\newcommand{\Cl}{\mathrm{Cl}}

\newcommand{\Disc}{\mathrm{Disc}}
\newcommand{\Aut}{\mathrm{Aut}}
\newcommand{\Avg}{\mathrm{Avg}}
\newcommand{\cO}{\mathcal{O}}
\newcommand{\bL}{\mathbb{L}}

\newtheorem{thm}{\bf{Theorem}}
\newtheorem{theorem}{\bf{Theorem}}[section]

\newtheorem{corollary}[theorem]{\bf{Corollary}}

\newtheorem{proposition}[theorem]{\bf{Proposition}}

\newtheorem{lemma}[theorem]{\bf{Lemma}}

\theoremstyle{definition}

\newtheorem{remark}[theorem]{\bf{Remark}}

\input xy
\xyoption{all}

\usepackage[OT2,T1]{fontenc}

\DeclareSymbolFont{cyrletters}{OT2}{wncyr}{m}{n}
\DeclareMathSymbol{\Sha}{\mathalpha}{cyrletters}{"58}

\title{Davenport-Heilbronn Theorems for Quotients of Class Groups}
\author{Zev Klagsbrun}
\address{Center for Communications Research, 4320 Westerra Court, San Diego, CA 92121}
\email{zdklags@ccrwest.org}
\begin{document}
\bibliographystyle{alpha}

\begin{abstract}
We prove a generalization of the Davenport-Heilbronn theorem to quotients of ideal class groups of quadratic fields by the primes lying above a fixed set of rational primes $S$. When restricted to quadratic fields in which all primes in $S$ split completely, our results are consistent with Cohen and Lenstra's model for such quotients.

Additionally, we obtain average sizes for the relaxed Selmer group $\Sel_3^S(K)$ and for  $\cO_{K,S}^\times/(\cO_{K,S}^\times)^3$ as $K$ varies among quadratic fields with a fixed signature ordered by discriminant.
\end{abstract}
\maketitle

\section{Introduction}

One of the few proven results concerning the distribution of class groups of number fields is the Davenport-Heilbronn theorem which states:
\begin{thm}[Theorem 3 in \cite{DH}]
\label{thm:Daveport-Heilbronn theorem}
When ordered by absolute discriminant,
\begin{enumerate}[(i)]
\item the average size of $\Cl(K)[3]$ as $K$ ranges over imaginary quadratic fields is equal to $2$ and 
\item the average size of $\Cl(K)[3]$ as $K$ ranges over real quadratic fields is equal to $4/3$.
\end{enumerate}
\end{thm}


We consider the case where $\Cl(K)$ is replaced with its quotient by the subgroup generated by the classes of primes lying above a fixed set of rational primes $S$. Explicitly, for a quadratic field $K$, define $\Cl(K)_S:=\Cl(K)/\langle S_K\rangle$, where $S_K$ is the set of primes of $\cO_K$ lying above the primes in $S$ and $\langle S_K \rangle$ is the subgroup of $\Cl(K)$ generated by the ideal classes of the primes in $S_K$.

One nominally expects that if each prime $p$ in $S$ splits completely as $\p \overline{\p}$ in $K$, then the ideal classes $[\p]$ for $p \in S$ should be distributed uniformly and independently at random in $\Cl(K)$. Conditioning on the distribution for $\Cl(K)$ given by the Cohen--Lenstra heuristics, this yields a prediction for the distribution of $\Cl(K)_S$. Indeed, when $S$ contains a single prime, this conjecture actually appears in the original work of Cohen and Lenstra on the distribution of class groups of quadratic fields \cite{CL} --- see Remark \ref{rem:CL remark}.
 
We prove that the average size of $\Cl(K)_S[3]$ as $K$ ranges over imaginary (resp. real) quadratic fields where all primes in $S$ split completely in $K$ is consistent with this conjectured distribution.

\begin{thm}
\label{thm:Daveport-Heilbronn theorem for quotients where S splits}
When ordered by absolute discriminant,
\begin{enumerate}[(i)]
\item \label{it:complex case}the average size of $\Cl(K)_S[3]$ as $K$ ranges over imaginary quadratic fields where all primes in $S$ split completely in $K$ is equal to $1 + 3^{-|S|}$ and 
\item \label{it:real case}the average size of $\Cl(K)_S[3]$ as $K$ ranges over real quadratic fields where all primes in $S$ split completely in $K$ is equal to $1 + 3^{-|S|-1}$.
\end{enumerate}
\end{thm}



If we don't condition on the splitting behavior of the primes in $S$, then we obtain the following:

\begin{thm}
\label{thm:Daveport-Heilbronn theorem for quotients}
When ordered by absolute discriminant,
\begin{enumerate}[(i)]
\item the average size of $\Cl(K)_S[3]$ as $K$ ranges over imaginary quadratic fields is equal to
\linebreak
$\displaystyle{
1 + 
\frac{1}{3^{|S|}}
\prod_{p \in S} 
\left(
2  
+
\frac{1}{p+1}
\right)
}$ and
\item the average size of $\Cl(K)_S[3]$ as $K$ ranges over real quadratic fields is equal to
\linebreak
$\displaystyle{
1 + 
\frac{1}{3^{|S|+1}}
\prod_{p \in S} 
\left(
2  
+
\frac{1}{p+1}
\right)
}$.
\end{enumerate}
\end{thm}



\begin{remark}
\label{rem:CL remark}
As noted above, the average sizes appearing in Theorem \ref{thm:Daveport-Heilbronn theorem for quotients where S splits} are consistent with modeling $\Cl(K)$ as a random group subject to the Cohen-Lenstra distribution and then taking the quotient of $\Cl(K)$ by a subgroup of $\Cl(K)$ generated by $|S|$ elements of $\Cl(K)$ chosen uniformly and independently at random. When $|S| = 1$ and $K$ is an imaginary quadratic field, then Cohen and Lenstra attribute to Dick Gross the observation that the distribution obtained in the manner is identical to the Cohen-Lenstra distribution for real quadratic fields \cite{CL}.
\end{remark}

\begin{remark}
\label{rem:CL remark mats}
The average size in part \eqref{it:complex case} (resp. \eqref{it:real case}) of Theorem \ref{thm:Daveport-Heilbronn theorem for quotients where S splits} is the same as the average size of the left-nullspace of a random $n \times n + |S|$ (resp. $n \times n + |S|+1$) matrix with random entries in $\FF_3$ as $n \rightarrow \infty$.
\end{remark}

The core idea underlying the proof of the Davenport-Heilbronn theorem is how to use the geometry of numbers to count cubic fields. The application to class groups is almost an afterthought that arises from a bijection originally due to Hasse between the set of index 3 subgroups of $\Cl(K)$ and the set of isomorphism classes of cubic fields $L$ with $d_L = d_K$ \cite{Hasse}.

We establish a similar bijection between the set of index 3 subgroups of $\Cl(K)_S$ and the set of isomorphism classes of cubic fields $L$ with $d_L = d_K$. Theorem \ref{thm:Daveport-Heilbronn theorem for quotients} is then proven using recent results of Bhargava, Shankar, and Tsimerman that count the number of cubic fields with bounded discriminant satisfying a set of local conditions \cite{Bhargava-Shankar-Tsimerman}. This counting may equivalently be accomplished by appealing to recent work of Taniguchi and Thorne \cite{TT}.




\subsection{Additional Results}

In addition to Theorem \ref{thm:Daveport-Heilbronn theorem for quotients}, we also prove similar distribution results for a pair of objects closely connected to $\Cl(K)_S[3]$.
A standard result (see Section 8.3.2 of \cite{Cohen-Number Theory vol 1}, for example) shows that $\Cl(K)_S[3]$ sits in the short exact sequence
\begin{equation}
\label{eq:Selmer sequence}
0 \rightarrow 
\cO_{K,S}^\times/(\cO_{K,S}^\times)^3 
\rightarrow 
\Sel_3^S(K)
\rightarrow
\Cl(K)_S[3]
\rightarrow 0,
\end{equation}
where $\cO_{K,S}^\times$ is the $S_K$ units of $\cO_K$ and $\Sel_3^S(K)$ is the 3-Selmer group of $K$ relaxed at $S_K$, defined as
\[
\Sel_3^S(K) := 
\left \{
\alpha \in K^\times/(K^\times)^3
:
\val_\p(\alpha) \equiv 0 \pmod{3}
\text{ for all }
\p \not \in S_K
\right \}.
\]
We are able to compute average sizes for both $\Sel_3^S(K)$ and $\cO_{K,S}^\times/(\cO_{K,S}^\times)^3$.

\begin{thm}
\label{thm:Daveport-Heilbronn theorem for Sel_3^S}
When ordered by absolute discriminant,
\begin{enumerate}[(i)]
\item the average size of $\Sel_3^S(K)$ as $K$ ranges over real quadratic fields is equal to 
\linebreak
$\displaystyle{3^{|S|} + 3^{|S|+1}\prod_{p \in S} \left( 1 + \frac{p}{p+1} \right)}$
and
\item the average size of $\Sel_3^S(K)$ as $K$ ranges over imaginary quadratic fields is equal to
$\displaystyle{3^{|S|} + 3^{|S|}\prod_{p \in S} \left( 1 +\frac{p}{p+1} \right)}$. 
\end{enumerate}
\end{thm}

\begin{thm}
\label{thm:Daveport-Heilbronn theorem for S-units}
When ordered by absolute discriminant,
\begin{enumerate}[(i)]
\item the average size of $\cO_{K,S}^\times/(\cO_{K,S}^\times)^3$ as $K$ ranges over real quadratic fields is equal to
$\displaystyle{3^{|S|+1}\prod_{p \in S} \left( 1 + \frac{p}{p+1} \right)}$
and
\item the average size of $\cO_{K,S}^\times/(\cO_{K,S}^\times)^3$ as $K$ ranges over imaginary quadratic fields is equal to
$\displaystyle{3^{|S|}\prod_{p \in S} \left( 1 + \frac{p}{p+1} \right)}$.
\end{enumerate}
\end{thm}

\begin{remark}
The results of Bhargava, Shankar, and Tsimerman in \cite{Bhargava-Shankar-Tsimerman} and Taniguchi and Thorne in \cite{TT} yield second order terms for the number of cubic fields with bounded discriminant satisfying a set of local conditions. As a result, it is possible to bounds the rate of convergence in Theorems \ref{thm:Daveport-Heilbronn theorem for quotients where S splits}, \ref{thm:Daveport-Heilbronn theorem for quotients}, \ref{thm:Daveport-Heilbronn theorem for Sel_3^S}, and \ref{thm:Daveport-Heilbronn theorem for S-units}.
\end{remark}

\subsection{Related Work}

Recent work of Varma \cite{Varma-RayClassGroups} generalized Theorem \ref{thm:Daveport-Heilbronn theorem} to ray class groups of fixed integral conductor. While the formulas appearing in Theorem 1 in \cite{Varma-RayClassGroups} are similar to those appearing in Theorem \ref{thm:Daveport-Heilbronn theorem for Sel_3^S} above, neither result appears to directly follow from the other.

In a companion paper, the author obtains variants of Theorems \ref{thm:Daveport-Heilbronn theorem for quotients where S splits}, \ref{thm:Daveport-Heilbronn theorem for quotients},  \ref{thm:Daveport-Heilbronn theorem for Sel_3^S}, and \ref{thm:Daveport-Heilbronn theorem for S-units} for quotients of class groups of cubic fields, considering 2-torsion in the class groups, rather than 3-torsion~\cite{DHforCubics}. Those results are similar in substance to the ones appearing here, but rely on a correspondence of Heilbronn between the unramified quadratic extensions of a cubic field $K$ and the isomorphism classes of quartic fields $L$ having cubic resolvent $K$ \cite{H}.

The results in Sections \ref{sec:CFT} -- \ref{sec:avg size} generalize results for the case $S = \{3\}$ that appear in Sections 9.1 -- 9.4 of the author's work with Jordan, Poonen, Skinner, and Zaytman on the distribution of K-groups of quadratic fields~\cite{JKPSZ}. While some of the proofs included here are substantively similar to those appearing in~\cite{JKPSZ}, we have nonetheless included them for the benefit of the reader.

Additionally, a variant of Theorem \ref{thm:Daveport-Heilbronn theorem for quotients where S splits}   when $S$ contains a single prime was obtained independently in unpublished work of Wood and is proved using similar methods \cite{Wood}.

\section*{Notation}
We will use the following notation throughout this paper:
\begin{itemize}
\item $S$ will be a set of rational primes.
\item $K$ will be a quadratic field.
\item If $F$ is a field, then $d_F$ is the discriminant of $F$.
\item $\cO_K$ will be the ring of integers of $K$.
\item $S_K$ will denote the set of primes of $\cO_K$ lying above primes in $S$.
\item $\cO_{K,S}^\times$ will denote the $S_K$-units of $K$.
\item $\Sel_3^S(K)$ will be the 3-Selmer group of $K$ relaxed at the primes in $S_K$.
\item $\Cl(K)$ will be the ideal class group of $\cO_K$.
\item $\Cl(K)_S$ will be the quotient $\Cl(K)/\langle S_K\rangle$, where $\langle S_K \rangle $ is the subgroup of $\Cl(K)$ generated by primes in $S_K$.
\item If $p$ is a rational prime, then $\bL_p$ will denote the unique unramified cubic extension of $\QQ_p$.
\end{itemize}

\section*{Acknowledgement}

I would like to thank Bjorn Poonen and Genya Zaytman for suggesting the proof of Proposition \ref{prop:DeloneFaddeevQuotient}. I would also like to thank Manjul Bhargava and Melanie Wood for valuable suggestions regarding the framing of Theorem \ref{thm:Daveport-Heilbronn theorem for quotients where S splits} with respect to the Cohen-Lenstra heuristics.

\section{Class Fields Theory}
\label{sec:CFT}

As remarked above, Davenport and Heilbronn relied on a bijection between the set of index 3 subgroups of $\Cl(K)$ and the set of isomorphism classes of cubic fields $L$ with $d_L = d_K$. We describe this correspondence, originally due to Hasse, in Proposition \ref{prop:DeloneFaddeev} before establishing a similar correspondence for $\Cl(K)_S$ in Proposition \ref{prop:DeloneFaddeevQuotient}.

\begin{proposition}[Satz 7 in \cite{Hasse}]
\label{prop:DeloneFaddeev}
The following are in bijective correspondence.
\begin{enumerate}[(i)]
\item \label{it:DF1}The set of index 3 subgroups of $\Cl(K)$.
\item \label{it:DF2}The set of unramified $\ZZ/3\ZZ$-extensions $M$ of $K$.
\item \label{it:DF3}The set of isomorphism classes of cubic fields $L$ with $d_L = d_K$.
\end{enumerate}
\end{proposition}

\begin{proof}
The correspondence \eqref{it:DF1} $\leftrightarrow$ \eqref{it:DF2} is class field theory. The correspondence \eqref{it:DF2} $\leftrightarrow$ \eqref{it:DF3} will follow from Satz 3 in \cite{Hasse}. 

Suppose that $M/K$ is unramified. Since the Hilbert class field $H_K$ is Galois over $\QQ$ and $M \subset H_K$, we get that $M/\QQ$ is Galois. We next observe that the non-trivial element $\sigma \in \Gal(K/\QQ)$ acts on $\Cl(K/\QQ)$ and therefore on $\Gal(H_K/\QQ)$ by inversion. As a result, $\Gal(K/\QQ)$ acts non-trivially on $\Gal(M/K)$, so we see that $\Gal(M/\QQ) \simeq \mathcal{S}_3$. The map \eqref{it:DF2}~$\rightarrow$~\eqref{it:DF3} sends $M/K$ to any of the 3 isomorphic cubic subfields $L$ of $M/\QQ$. By Satz 3 in~\cite{Hasse}, $d_L = \mathbf{N}_{K/\QQ}(\mathfrak{f}) d_K$, where $\mathfrak{f}$ is the conductor of $M/K$. Since $M/K$ is unramified, we find that $d_L = d_K$.

For the opposite direction, let $M$ be the Galois closure of $L/\QQ$. Since $L/\QQ$ is nowhere totally ramified, $M$ must be an $\mathcal{S}_3$ extension of $\QQ$. As such, $M$ contains a unique quadratic field $K$ and the map \eqref{it:DF3}~$\rightarrow$~\eqref{it:DF2} sends $L/\QQ$ to the extension $M/K$. Once again by Satz 3 in \cite{Hasse}, $d_L = \mathbf{N}_{K/\QQ}(\mathfrak{f}) d_K$. Since $d_L = d_K$, we have $\mathfrak{f} = \cO_K$, so $M/K$ is unramified. Since the maps \eqref{it:DF2} $\rightarrow$ \eqref{it:DF3} and \eqref{it:DF3} $\rightarrow$ \eqref{it:DF2} can be seen to invert each other, the result follows.
\end{proof}

\begin{proposition}
\label{prop:DeloneFaddeevQuotient}
The following are in bijective correspondence.
\begin{enumerate}[(i)]
\item The set of index 3 subgroups of $\Cl(K)_S$.
\item The set of unramified $\ZZ/3\ZZ$-extensions $M$ of $K$ in which all primes in $S_K$ split completely.
\item The set of isomorphism classes of cubic fields $L$ with $d_L = d_K$ such that for all $p \in S$, if $d_K \in (\QQ_p^\times)^2$, then $p$ splits completely in $L$.
\end{enumerate}
\end{proposition}
\begin{proof}
By class field theory, $\Cl(K)_S$ is the Galois group of the maximal unramified abelian extension $H_S$ of $K$ in which $\Frob_\p \in \Gal(H_S/K)$ is trivial for all $\p \in S_K$.
As a result, the set of index 3 subgroups of $\Cl(K)_S$ is in bijective correspondence with set of unramified $\ZZ/3\ZZ$-extensions $M$ of $K$ in which all primes in $S_K$ split completely.

The maps for the correspondence between \eqref{it:DF2} and \eqref{it:DF3} are the same as in Proposition \ref{prop:DeloneFaddeev}.
We only need to show that the images satisfy the properties claimed.
The map \eqref{it:DF2} $\rightarrow$ \eqref{it:DF3} sends $M/K$ to any of the cubic subfields of the $\mathcal{S}_3$-extension $M/\QQ$.
If $d_K \in (\QQ_p^\times)^2$, then $p$ splits completely in $K/\QQ$ and each prime $\p \mid p$ splits completely in $M/K$ by assumption.
As a result, $p$ splits completely in $L/\QQ$.

For the reverse map \eqref{it:DF3} $\rightarrow$ \eqref{it:DF2}, send $L/\QQ$ to the extension $M/K$ where $M$ is the Galois closure of $L/\QQ$ and $K \subset M$ is the unique quadratic subextension of $M$.
We wish to show that all primes in $S_K$ split completely in $M/K$.

Suppose that $p \in S$. If $d_K \in (\QQ_p^\times)^2$, then $p$ splits in $K$ and by assumption $p$ splits completely in $L$ as well.
As a result, $p$ splits completely in $M/\QQ$ and therefore in $M/K$ as well.
If $d_K \not \in (\QQ_p^\times)^2$, then $p$ has two primes lying above it in $L$.
The Galois structure of $M/\QQ$ then forces $p$ to have three primes lying above it in $M$.
Since both $K/\QQ$ and $M/K$ are Galois, we find that $p$ does not split in $K/\QQ$ and that all primes above $p$ in $K$ split completely in $M/K$.
\end{proof}

\begin{corollary}
\label{cor:Delone-Faddeev for quotients}
If $K$ is a quadratic field, then
\[
\left |
\Cl(K)_S[3]
\right |
=
1 +
2 \left |
\left \{
\text{cubic fields }
L
\text{ with }
d_L = d_K
\text{ such that }
L \otimes \QQ_p \ne \bL_p
\text{ for all }
p \in S
\right \}
\right |
\]
\end{corollary}
\begin{proof}
Since $\Cl(K)_S$ is a finite abelian group, the number of index three subgroups of $\Cl(K)_S$ is equal to the number of order three subgroups of $\Cl(K)_S$.
By Proposition \ref{prop:DeloneFaddeevQuotient}, this is the same as the number of cubic fields $L$ with $d_L = d_K$ such that $L \otimes \QQ_p \ne \bL_p$.
The result follows since each order three subgroup contains two non-trivial elements and any two distinct order three subgroups intersect in the trivial group.
\end{proof}

\section{Counting Fields}

In order to obtain our main results, we will need the following two theorems that count the number of quadratic and cubic fields satisfying a finite set of local conditions.

\begin{theorem}
\label{thm:numquadfieldsforS_1}
Let $S_1 \subset S$.
Then the number of real (resp. imaginary) quadratic fields $K$ with $|d_K| < X$ such that all primes in $S_1$ split in $K/\QQ$ and all primes in $S \setminus S_1$ do not split in $K/\QQ$ is equal to 
\[
\frac{1}{2\zeta(2)}
\left (
\prod_{p \in S_1} 
\frac{p}{2(p+1)}
\right )
\left (
\prod_{p \in S \setminus S_1}
\frac{p+2}{2(p+1)}
\right )
\cdot
X
+ o(X)
\]
\end{theorem}

To prove Theorem \ref{thm:numquadfieldsforS_1} we will need to following lemma.

\begin{lemma}
\label{lem:equidistribution of fundamental discriminants}
Let $\mathfrak{d}$ be an odd squarefree number. Then for any class $d_0 \pmod{16\mathfrak{d}^2}$ that occurs as a fundamental discriminant, we have
\[
\left |
\{
0 < d < X :
d \text{ fundamental and }
d \equiv d_0
\pmod{\mathfrak{16d^2}}
\}
\right |
=
\frac{1}{2\zeta(2)}
\prod_{p \mid 2\mathfrak{d}}
\frac{1}{p^2 - 1}
\cdot
X
+ O(\sqrt{X})
\]
and
\[
\left |
\{
-X < d < 0 :
d \text{ fundamental and }
d \equiv d_0
\pmod{\mathfrak{16d^2}}
\}
\right |
=
\frac{1}{2\zeta(2)}
\prod_{p \mid 2\mathfrak{d}}
\frac{1}{p^2 - 1}
\cdot
X
+ O(\sqrt{X})
\]
\end{lemma}
\begin{proof}
By Corollary 1 to Theorem 1 in \cite{CR62}, for any squarefree $d_0$, we have
\[
\left |
\{
0 < d < X :
d \text{ is squarefree and }
d \equiv d_0 \pmod{\mathfrak{4d^2}}
\}
\right |
=
\frac{1}{\zeta(2)}
\prod_{p \mid 2\mathfrak{d}}
\frac{1}{p^2 - 1}
\cdot
X
+ O(\sqrt{X}).
\]
The result follows from considering each possible squarefree class modulo 4.
\end{proof}

\begin{proof}[Proof of Theorem \ref{thm:numquadfieldsforS_1}]
A prime $p$ splits in a quadratic field $K/\QQ$ if and only if $d_K \in (\QQ_p^\times)^2$. Letting  $\mathfrak{d}$ be the product of all odd primes in $S$, the splitting type in $K/\QQ$ of all primes in $S$ is then determined by the class of $d_K \pmod{16\mathfrak{d}^2}$.

For odd $p$, the number of non-trivial square classes modulo $p^2$ is equal to $\frac{p^2 -p}{2}$ and the number of non-trivial square classes modulo $16$ is 2.
Therefore, the number of admissable classes $d \pmod{16\mathfrak{d}^2}$ such that if $d_K \equiv d \pmod{16\mathfrak{d}^2}$, then all primes of $S_1$ split in $K/\QQ$ and all primes in $S \setminus S_1$ do not split in $K/\QQ$ is given by

\begin{equation}
\label{eq:num admissable classes}
c_2
\prod_{p \in S_1}
\frac{p^2 - p}{2}
\prod_{p \in S\setminus S_1}
\frac{p^2 + p - 2}{2}.
\end{equation}
where 
$c_2
=
\left \{
\begin{matrix}[ll]
2 & \text{if } 2 \in S \text{ and}\\
6 & \text{if } 2 \not \in S.
\end{matrix}
\right .
$

By Lemma \ref{lem:equidistribution of fundamental discriminants}, the number of fundamental discriminants in each class
is given by
\begin{equation}
\label{eq:num fundamental discriminants per admissable class}
\frac{1}{2\zeta(2)}
\prod_{p \mid 2\mathfrak{d}}
\frac{1}{p^2 - 1}
\cdot
X
+ O(\sqrt{X}).
\end{equation}
Multiplying \eqref{eq:num admissable classes} by \eqref{eq:num fundamental discriminants per admissable class} and accounting for whether or not $2 \in S$ gives the result.
\end{proof}

\begin{remark}
\label{rem:numquadfields}
By taking $S = \emptyset$, we recover the standard asymptotic for the number of quadratic fields with bounded discriminant.
\end{remark}

In order to count cubic fields, we will use the following specialization of a theorem of Bhargava, Shankar, and Tsimerman.

\begin{theorem}[Theorem 8 in \cite{Bhargava-Shankar-Tsimerman}]
\label{thm:theorem 8 in BST}
Let $S$ be a finite set of primes. For each prime $p \in S$, let $\Sigma_p$ be a set of maximal cubic orders over $\ZZ_p$ that are not totally ramified. Then the number of totally real (resp. complex) cubic fields $L$ with $|d_L|<X$ that are nowhere totally ramified and have $\cO_L \otimes \ZZ_p \in \Sigma_p$ for all $p$ is given by

\begin{equation}
\label{eq:cubic field count with local conds}
N(X) = \frac{c_\infty}{\zeta(2)}
\left (
\prod_{p \in S}
\frac{p}{p+1}
\sum_{R \in \Sigma_p}
\frac{1}{|\Aut(R)|}
\cdot
\frac{1}{\Disc_p(R)}
\right )
\cdot
X
+ o(X),
\end{equation}
where $c_\infty = \frac{1}{12}$ (resp. $c_\infty = \frac{1}{4}$) and $\Disc_p(R)$ is the $p$-part of  the discriminant of $R$.
\end{theorem}
\begin{proof}
Following the convention in \cite{Bhargava-Shankar-Tsimerman}, for all primes $p \not \in S$, set $\Sigma_p$ to be the set of all maximal cubic orders over $\ZZ_p$ that are not totally ramified.
 
By Theorem 8 in \cite{Bhargava-Shankar-Tsimerman}, the number of totally real (resp. complex) cubic fields $L$ with $|d_L|<X$ and $\cO_L \otimes \ZZ_p \in \Sigma_p$ for all $p$ is given by
\begin{equation}
\label{eq:cubic field count}
N(X) = c_\infty
\left (
\prod_p
\frac{p-1}{p}
\sum_{R \in \Sigma_p}
\frac{1}{|\Aut(R)|}
\cdot
\frac{1}{\Disc_p(R)}
\right )
\cdot
X
+ o(X),
\end{equation}
where $c_\infty$ is as above.
The choice of $\Sigma_p$ for $p \not \in S$ ensures that we count exactly those fields that are nowhere totally ramified.

Letting $\Gamma_p$ be the set of all maximal cubic orders over $\ZZ_p$ that are not totally ramified, it is an easy exercise to see that
$\displaystyle{
\sum_{R \in \Gamma_p}
\frac{1}{|\Aut(R)|}
\cdot
\frac{1}{\Disc_p(R)}
=
\frac{p+1}{p}
}$.
As a result, for all $p \not \in S$, we have
$\displaystyle{
\sum_{R \in \Sigma_p}
\frac{1}{|\Aut(R)|}
\cdot
\frac{1}{\Disc_p(R)}
=
\frac{p+1}{p}
}$.

Letting 
$\displaystyle{
S_\Sigma
=
\prod_p
\frac{p-1}{p}
\sum_{R \in \Sigma_p}
\frac{1}{|\Aut(R)|}
\cdot
\frac{1}{\Disc_p(R)}
}$,
we have
\begin{multline*}
S_\Sigma
=
\prod_{p \not \in S}
\left (
1 - \frac{1}{p^2}
\right )
\prod_{p \in S}
\frac{p-1}{p}
\sum_{R \in \Sigma_p}
\frac{1}{|\Aut(R)|}
\cdot
\frac{1}{\Disc_p(R)}
\\
=
\prod_p
\left (
1 - \frac{1}{p^2}
\right ) 
\prod_{p \in S}
\frac{p}{p+1}
\sum_{R \in \Sigma_p}
\frac{1}{|\Aut(R)|}
\cdot
\frac{1}{\Disc_p(R)}
=
\frac{1}{\zeta(2)}
\prod_{p \in S}
\frac{p}{p+1}
\sum_{R \in \Sigma_p}
\frac{1}{|\Aut(R)|}
\cdot
\frac{1}{\Disc_p(R)}
\end{multline*}
Combined with \eqref{eq:cubic field count}, this gives \eqref{eq:cubic field count with local conds}.
\end{proof}

\section{Average Sizes of Quotients of Class Groups}
\label{sec:avg size}
%
By Corollary \ref{cor:Delone-Faddeev for quotients}, in order to count elements of $\Cl(K)_S[3]$, we want to count cubic fields $L$ such that $\cO_L \otimes \ZZ_p \ne \cO_{\bL_p}$ for all $p \in S$. We are able to do this using Theorem \ref{thm:theorem 8 in BST}.

\begin{theorem}\text{ }
\label{thm:Application of Theorem 8}
The number of nowhere totally ramified totally real (resp. complex) cubic fields $L$ with 
$|d_L| < X$
and
$\cO_L \otimes \ZZ_p \ne \cO_{\bL_p}$
for all $p \in S$
is
\[
\frac{c_\infty}{3^{|S|}\zeta(2)}
\prod_{p\in S}
\left (
2 
+
\frac{1}{p+1}
\right )
+o(X).
\]
where $c_\infty = \frac{1}{12}$ (resp. $c_\infty = \frac{1}{4}$).
\end{theorem}
\begin{proof}
For each prime $p \in S$, let $\Gamma_p$ be the set of all maximal cubic orders over $\ZZ_p$ that are not totally ramified and set $\Sigma_p = \Gamma_p \setminus \{ \cO_{\bL_p} \}$.

For each $p \in S$, we then have 
\[
\frac{p}{p+1} \sum_{R \in \Sigma_p}
\frac{1}{|\Aut(R)|}
\cdot
\frac{1}{\Disc_p(R)}
=
\frac{p}{p+1} 
\left (
\frac{p+1}{p} 
-
\frac{1}{3} \right )
=
1 - \frac{p}{3(p+1)}
=
\frac{1}{3}
\left (
2
+
\frac{1}{p+1}
\right)
\]
The result then follows from Theorem \ref{thm:theorem 8 in BST}.
\end{proof}

\begin{proof}[Proof of Theorem \ref{thm:Daveport-Heilbronn theorem for quotients}]
Combining
Theorem \ref{thm:Application of Theorem 8}
with
Corollary \ref{cor:Delone-Faddeev for quotients},
the total number of non-trivial elements of $\Cl(K)_S[3]$ for all real (resp. imaginary) quadratic fields $K$ with $|d_K| < X$ is given by
$\displaystyle{
\frac{c_\infty}{3^{|S|}\zeta(2)}
\prod_{p \in S}
\left (
2
+
\frac{1}{p+1}
\right )
\cdot X
+ o(X)
}$,
where $c_\infty = \frac{1}{6}$
(resp. $c_\infty = \frac{1}{2}$).
As noted in Remark \ref{rem:numquadfields}, the total number of real (resp. imaginary) quaratic fields $K$ with $|d_K| < X$ is equal to
$
\frac{1}{2\zeta(2)}
\cdot 
X
+ o(X)
$.
We therefore get that the average number of non-trivial elements of
$\Cl(K)_S[3]$ for all real (resp. imaginary)
quadratic fields $K$ with $|d_K|<X$
is equal to 
$\displaystyle{
\frac{c_\infty^\prime}{3^{|S|}}
\cdot 
\prod_{p \in S}
\left (
2
+
\frac{1}{p+1}
\right )
}$,
where
$c_\infty^\prime = 1/3$
(resp.
$c_\infty^\prime = 1$).
\end{proof}

\section{Averages for $S$-Unit Groups}

We now want to consider the average size of $\cO_{K,S}^\times/\left (\cO_{K,S}^\times \right)^3$ as $K$ varies over real (resp. imaginary) quadratic fields.
For uniformity, we will assume that $K \ne \QQ(\sqrt{-3})$ and note that this does not affect the averages.

\begin{lemma}
\label{lem:Dirichlet}
Let $K$ be a quadratic field and let $S_1 \subset S$ be the set of primes in $S$ that split in $K/\QQ$.
Then
$\left |
\cO_{K,S}^\times/\left (\cO_{K,S}^\times \right)^3
\right |
= 
3^{r_\infty + |S| + |S_1|}$,
where
$\displaystyle{
r_\infty
=
\left \{
\begin{matrix}[ll]
1 & \text{if } K \text{ is real and}\\
0 & \text{if } K \text{ is imaginary.}
\end{matrix}
\right .
}$
\end{lemma}

\begin{proof}
This follows directly from Dirichlet's unit theorem.
\end{proof}

\begin{lemma}
\label{lem:Avg3^|S_1|}
For each quadratic field $K$, let $S_1 = S_1(K)$ be the set of primes in $S$ that split in $K/\QQ$. Then as $K$ varies over real (resp. imaginary) quadratic fields ordered by discriminant, we have
\begin{equation}
\label{eq:Avg 3^|S_1|}
\Avg \left ( 3^{|S_1|} \right)
=
\prod_{p \in S}
\left (
1 + \frac{p}{p+1}
\right ).
\end{equation}
\end{lemma}

\begin{proof}
We proceed by induction.
Suppose that $S = \{ p \}$.
Then $p$ splits in $K$ exactly when
$d_K \in ( \QQ_p^\times ) ^2$.
As $d_K$ ranges over fundamental discriminants, the proportion of $d_K$ such that
$d_K \in ( \QQ_p^\times ) ^2$
is equal to
$\dfrac{p}{2(p+1)}$.
We therefore have
\[
\Avg \left ( 3^{|S_1|} \right)
=
1 \cdot \frac{p+2}{2(p+1)}
+
3 \cdot \frac{p}{2(p+1)}
=
\frac{4p+2}{2(p+1)}
=
1 + \frac{p}{p+1}.
\]

Now suppose that \eqref{eq:Avg 3^|S_1|} holds for a set $S$ and let $S^\prime = S \cup \{ q \}$ for some prime $q \not \in S$.
We then have
\[
\Avg \left ( 3^{|S_1|} \right)
=
\rho
\cdot 
\Avg 
\left (
3^{|S_1|}
\vert
q \in S_1
\right )
+
(1 - \rho)
\cdot 
\Avg 
\left (
3^{|S_1|}
\vert
q \not \in S_1
\right ),
\]
where $\rho$ is the probability that $q \in S_1$.
By the inductive hypothesis, we have
\linebreak
$\displaystyle{
\Avg 
\left (
3^{|S_1|}
\vert
q \in S_1
\right )
=
3 \prod_{p \in S}
\left (
1 + \frac{p}{p+1}
\right )
}$
and
$\displaystyle{
\Avg 
\left (
3^{|S_1|}
\vert
q \not \in S_1
\right )
=
\prod_{p \in S}
\left (
1 + \frac{p}{p+1}
\right )
}$.
Since the probability that $q \in S_1$ is given by $\dfrac{q}{2(q+1)}$, we have
\[
\Avg \left ( 3^{|S_1|} \right)
=
1 \cdot 
\frac{q+2}{2(q+1)}
+
3 \cdot
\frac{q}{2(q+1)}
\prod_{p \in S}
\left (
1 + \frac{p}{p+1}
\right )
=
\prod_{p \in S^\prime}
\left (
1 + \frac{p}{p+1}
\right )
\]
\end{proof}

\begin{proof}[Proof of Theorem \ref{thm:Daveport-Heilbronn theorem for S-units}]
By Lemma \ref{lem:Dirichlet}, we have
\[
\Avg
\left (
\left |
\cO_{K,S}^\times/\left (\cO_{K,S}^\times \right)^3
\right |
\right )
=
\Avg
\left (
3^{r_\infty + |S|+|S_1|}
\right )
=
3^{r_\infty + |S|}
\cdot
\Avg
\left (
3^{|S_1|}
\right )
\]
The result then follows from Lemma \ref{lem:Avg3^|S_1|}.
\end{proof}

\section{Average Sizes of Selmer Groups}

\begin{lemma}
\label{lem:size of Sel_3^S}
Let $K$ be a quadratic field and let $S_1 \subset S$ be the set of primes that split in $K/\QQ$.
Then
$\left |\Sel_3^S(K) \right|
=
3^{r_\infty + |S| + |S_1|}
\cdot
\left |
\Cl(K)_S[3]
\right|$
where $r_\infty$ is as in Lemma \ref{lem:Dirichlet}.
\end{lemma}
\begin{proof}
This follows from Lemma \ref{lem:Dirichlet} combined with \eqref{eq:Selmer sequence}.
\end{proof}

\begin{lemma}
\label{lem:avg Cl[3] for fixed S_1}
Let $S_1 \subset S$.
Then as $K$ varies through all real (resp. imaginary) quadratic fields (ordered by absolute discriminant) such that such that all primes in $S_1$ split in $K/\QQ$ and all primes in $S \setminus S_1$ do not split in $K/\QQ$, we have
$\displaystyle{
\Avg
\left (
\left |
\Cl(K)_S[3]
\right |
\right)
=
1 + \frac{c_\infty^\prime}{3^{|S_1|}}
}$,
where
$\displaystyle{
c_\infty^\prime
=
\left \{
\begin{matrix}[ll]
1/3 & \text{if } K \text{ is real and}\\
1 & \text{if } K \text{ is imaginary.}
\end{matrix}
\right .
}$
\end{lemma}

\begin{proof}
This will be a consequence of Theorem \ref{thm:theorem 8 in BST}. For $p \in S_1$, set $\Sigma_p = \{\Z_p^3 \}$ and for $p \in S \setminus S_1$, set $\Sigma_p = \{\Z_p \times \cO_{K_\p} :   K_{\p}/\QQ_p \text{ quadratic} \}$.
Observe that if $L$ is a nowhere totally ramified cubic field such that
$\cO_L \otimes \ZZ_p \in \Sigma_p$
for all $p \in S$, then the set of primes in $S$ for which 
$(d_L \in \QQ_p^\times)^2$ is equal to $S_1$.

We then have
\[
\frac{p}{p+1}
\sum_{R \in \Gamma_p}
\frac{1}{|\Aut(R)|}
\cdot
\frac{1}{\Disc_p(R)}
=
\left
\{
\begin{matrix}[ll]
\frac{p}{6(p+1)} & p \in S_1 \\
\frac{p+2}{2(p+1)}& p \in S \setminus S_1. 
\end{matrix}
\right .
\]

As a result, by Theorem \ref{thm:theorem 8 in BST}, the number of nowhere totally ramified totally real (resp. complex) cubic fields $L$ with $|d_L| < X$ such that
$\cO_L \otimes \ZZ_p \in \Sigma_p$
for all $p \in S$ is given by
\[
\frac{c_\infty}{\zeta(2)}
\prod_{p \in S_1}
\frac{p}{6(p+1)}
\prod_{p \in S \setminus S_1}
\frac{p+2}{2(p+1)} 
\cdot 
X
+ o(X),
\]
where $c_\infty = \frac{1}{12}$ (resp. $c_\infty = \frac{1}{4}$).

By Corollary \ref{cor:Delone-Faddeev for quotients}, we therefore get that the total number of non-trivial elements of $\Cl(K)_S$ as $K$ varies over all real (resp. imaginary) quadratic fields such that all primes in $S_1$ split in $K/\QQ$ and all primes in $S \setminus S_1$ do not split in $K/\QQ$ is equal to
\[
\frac{2c_\infty}{\zeta(2)}
\prod_{p \in S_1}
\frac{p}{6(p+1)}
\prod_{p \in S \setminus S_1}
\frac{p+2}{2(p+1)} 
\cdot 
X
+ o(X),
\]

By Theorem \ref{thm:numquadfieldsforS_1}, the number of such real (resp. imaginary) quadratic fields is equal to
\[
\frac{1}{2\zeta(2)}
\prod_{p \in S_1} 
\frac{p}{2(p+1)}
\prod_{p \in S \setminus S_1}
\frac{p+2}{2(p+1)}
\cdot
X
+ o(X)
\]
and as a result,
the average number of non-trivial elements in $\Cl(K)_S[3]$ for as $K$ ranges over real (resp. imaginary) quadratic fields such that all primes in $S_1$ split in $K/\QQ$ and all primes in $S \setminus S_1$ do not split in $K/\QQ$ is given by
$\displaystyle{
4 \cdot c_\infty
\prod_{p \in S_1} \frac{1}{3}
} = \frac{c_\infty^\prime}{3^{|S_1|}}.$
\end{proof}

\begin{remark}
It is worth noting that the average size of $|\Cl(K)_S[3]|$ in Lemma \ref{lem:avg Cl[3] for fixed S_1} depends only on $|S_1|$ and not on the actual primes contained in $S$ or $S_1$.
\end{remark}

\begin{proof}[Proof of Theorem \ref{thm:Daveport-Heilbronn theorem for quotients where S splits}]
This is Lemma \ref{lem:avg Cl[3] for fixed S_1} with $S = S_1$.
\end{proof}

\begin{theorem}
\label{thm:avg Sel_3 for fixed S_1}
Let $S_1 \subset S$.
Then as $K$ varies through all real (resp. imaginary) quadratic fields (ordered by absolute discriminant) such that all primes in $S_1$ split in $K/\QQ$ and all primes in $S \setminus S_1$ do not split in $K/\QQ$, we have
$\displaystyle{
\Avg
\left (
\left |
\Sel_3^S(K)
\right |
\right)
=
3^{r_\infty + |S_1| + |S|}
+
3^{|S|}
}$,
where $r_\infty$ is as in Lemma \ref{lem:Dirichlet}.
\end{theorem}

\begin{proof}
By Lemma \ref{lem:size of Sel_3^S}, we have
\[
\Avg
\left (
\left |
\Sel_3^S(K)
\right |
\right)
=
\Avg
\left (
3^{r_\infty + |S| + |S_1|}
\cdot
\left |
\Cl(K)_S[3]
\right |
\right)
=
3^{r_\infty + |S| + |S_1|}
\cdot
\Avg
\left (
\left |
\Cl(K)_S[3]
\right |
\right)
,
\]
where the second equality follows from the fact that $S_1$ is fixed.
The result then follows from Applying Lemma \ref{lem:avg Cl[3] for fixed S_1} and observing that with $c_\infty^\prime$ as in Lemma \ref{lem:avg Cl[3] for fixed S_1}, we have $3^{r_\infty}\cdot c_\infty^\prime = 1$.
\end{proof}

\begin{proof}[Proof of Theorem \ref{thm:Daveport-Heilbronn theorem for Sel_3^S}]
By Theorem \ref{thm:avg Sel_3 for fixed S_1}, for each $S_1 \subset S$,
$\displaystyle{
\Avg
\left (
\left |
\Sel_3^S(K)
\right |
\right)
=
3^{r_\infty + |S_1| + |S|}
+
3^{|S|}
}$
when we vary over real (resp. imaginary) quadratic fields $K$ such that all primes in $S_1$ split in $K/\QQ$ and all primes in $S \setminus S_1$ do not split in $K/\QQ$.
Therefore, if we vary over all real (resp. imaginary) quadratic fields $K$, we have
\[
\Avg
\left (
\left |
\Sel_3^S(K)
\right |
\right)
=
\sum_{S_1 \subset S}
\rho(S_1)
\left (
3^{r_\infty + |S_1| + |S|}
+
3^{|S|}
\right),
\]
where $\rho(S_1)$ is the probability that all of the primes in $S_1$ split in $K/\QQ$ and all primes in $S \setminus S_1$ do not split in $K/\QQ$ as $K$ varies.
However,
\begin{multline*}
\sum_{S_1 \subset S}
\rho(S_1)
\left (
3^{r_\infty + |S_1| + |S|}
+
3^{|S|}
\right)
\\
=
3^{r_\infty + |S|}
\sum_{S_1 \subset S}
\rho(S_1)
3^{|S_1|}
+
3^{|S|}
\sum_{S_1 \subset S}
\rho(S_1)
=
3^{r_\infty + |S|}
\Avg
\left (
3^{|S_1|}
\right)
+
3^{|S|},
\end{multline*}
so by Lemma \ref{lem:Avg3^|S_1|},
we get that
$\displaystyle{
\Avg
\left (
\left |
\Sel_3^S(K)
\right |
\right )
=
3^{r_\infty + |S|}
\prod_{p \in S}
\left (
1 + \frac{p}{p+1}
\right )
+
3^{|S|}
}$.
\end{proof}


\begin{thebibliography}{99}
\bibitem{Bhargava-Shankar-Tsimerman}M. Bhargava, A. Shankar, and J. Tsimerman. \textit{On the Davenport-Heilbronn theorems and second order terms}. Inventiones mathematicae. \textbf{193} (2013): 439--499.
\bibitem{CR62} E.Cohen and R. Robinson. \textit{On the distribution of the k-
free integers in residue classes}. Acta Arithmetica. \textbf{3} (1963): 283--93.
\bibitem{Cohen-Number Theory vol 1} H. Cohen. \textit{Number theory: Volume I: Tools and diophantine equations}. \textbf{239} Springer Verlag, 2009.
\bibitem{CL} H. Cohen and H. Lenstra Jr. \textit{Heuristics on class groups of number fields}. Number Theory Noordwijkerhout 1983 : 33--62.
\bibitem{DH} H. Davenport and H. Heilbronn. \textit{On the density of discriminants of cubic fields II}. Proceedings of the Royal Society of London. Series A, Mathematical and Physical Sciences (1971): 405-420.
\bibitem{Hasse} H. Hasse. \textit{Arithmetische Theorie der kubischen Zahlk\"{o}rper auf klassenk\"{o}rpertheoretischer Grundlage.} Mathematische Zeitschrift \textbf{31} (1930): 565--582.
\bibitem{H} H. Heilbronn. \textit{On the 2-classgroup of cubic fields}. Studies in Pure Mathematics. (1971): 117-119. 
\bibitem{JKPSZ} B. Jordan, Z. Klagsbrun, B. Poonen, C. Skinner, and Y. Zaytman. \textit{Statistics of $K$-groups modulo $p$ for the ring of integers of a varying quadratic number fields}. arXiv preprint 1703.00108 (2017).
\bibitem{DHforCubics} Z. Klagsbrun. \textit{The average sizes of two-torsion subgroups in quotients of class groups of cubic fields}. arXiv preprint 1701.02838.
\bibitem{TT} T. Taniguchi and F. Thorne. \textit{Secondary terms in counting functions for cubic fields.} Duke Mathematical Journal \textbf{162} (2013): 2451--2508.
\bibitem{Varma-RayClassGroups} I. Varma. \textit{The mean number of 3-torsion elements in ray class groups of quadratic fields}. https://arxiv.org/abs/1609.02292.

\bibitem{Wood} M. Wood. \textit{Cohen-Lenstra heuristics and local conditions}. Preprint available at  \linebreak https://www.math.wisc.edu/~mmwood/Publications.

\end{thebibliography}
\end{document}